\documentclass[10pt,reqno]{amsart}
\usepackage{amsmath}
\usepackage{amssymb}
\usepackage{amsthm}
\usepackage{epsfig}
\usepackage{mathrsfs}

\usepackage{tikz}

\newlength{\defbaselineskip}
\newcommand{\setlinespacing}[1]%
           {\setlength{\baselineskip}{#1 \defbaselineskip}}

\numberwithin{equation}{section}

\newtheorem{thm}{Theorem}[section]

\newtheorem{lem}[thm]{Lemma}
\newtheorem{prop}[thm]{Proposition}

\theoremstyle{definition}

\theoremstyle{remark}
\newtheorem{rem}[thm]{Remark}
\numberwithin{equation}{section}

\begin{document}
\title[Strichartz estimates in Wiener amalgam spaces]{Strichartz estimates for the Dirac flow in Wiener amalgam spaces}

\author{Seongyeon Kim, Hyeongjin Lee and Ihyeok Seo}

\thanks{This work was supported by a KIAS Individual Grant (MG082901) at Korea Institute for Advanced Study and the POSCO Science Fellowship of POSCO TJ Park Foundation (S. Kim), and by NRF-2022R1A2C1011312 (I. Seo).}


\address{School of Mathematics, Korea Institute for Advanced Study, Seoul 02455, Republic of Korea}
\email{synkim@kias.re.kr}

\address{Department of Mathematics, Sungkyunkwan University, Suwon 16419, Republic of Korea}
\email{hjinlee@skku.edu}
\email{ihseo@skku.edu}

\begin{abstract}
In this paper we obtain some new Strichartz estimates for the Dirac flow in the context of Wiener amalgam spaces which control the local regularity of a function and its decay at infinity separately unlike $L^p$ spaces. 
While it is well understood recently for some flows such as the Schr\"odinger and wave flows that work in the non-relativistic regime, nothing is known about the Dirac flow which governs a physical system in the case of relativistic fields.
\end{abstract}

\maketitle

\section{Introduction}		
In relativistic quantum mechanics the state of a free electron is represented by 
a wave function  $u(x,t):\mathbb{R}^{3+1} \rightarrow \mathbb{C}^{4}$ governed by the following Dirac equation:
\begin{equation}\label{DE}
\begin{cases}
-i\partial_{t}u+\mathcal{D}u+m\beta u=0,\\
u(x,0)=f(x),
\end{cases}
\end{equation}
where $m>0$ is the mass of the electron and the Dirac operator $\mathcal{D}$ is the first order operator defined as
\begin{equation*}
\mathcal{D} =-i\sum_{j=1}^{3}\alpha_{j}\partial_{j}=-i\alpha\cdot \nabla
\end{equation*}
with the $4\times4$ complex matrices (called Dirac matrices), 
\begin{equation*}
\alpha_{j}=
\begin{pmatrix}
0 & \sigma_{j}\\
\sigma_{j} & 0
\end{pmatrix},
\quad
\beta=
\begin{pmatrix}
I_{2} & 0\\
0 & -I_{2}
\end{pmatrix},
\end{equation*}
that are constructed from the Pauli matrices 
\begin{equation*}
\sigma_{1}=
\begin{pmatrix}
0 & 1\\
1 & 0
\end{pmatrix},
\quad
\sigma_{2}=
\begin{pmatrix}
0 & -i\\
i & 0
\end{pmatrix},
\quad
\sigma_{3}=
\begin{pmatrix}
1 & 0\\
0 & -1
\end{pmatrix}.
\end{equation*}
Here the Dirac matrices satisfy the following relations for $j,k=1,2,3$ 
\begin{equation}\label{matid}
\beta^{2}=\alpha_{j}^{2}=I_{4}, \quad \alpha_{j}\beta+\beta\alpha_{j}=0, \quad \alpha_{j}\alpha_{k}+\alpha_{k}\alpha_{j}=2\delta_{jk}I_{4}.
\end{equation}
See for example \cite{Th} for a more detailed introduction to the Dirac equation.

By scaling $u(\frac{x}{m},\frac{t}{m}) \rightarrow u(x,t)$, we shall always assume $m=1$ and consider the Dirac flow $e^{-it(\mathcal{D}+\beta)}$ that gives the solution 
$u(x,t)=e^{-it(\mathcal{D}+\beta)}f(x)$ to the Cauchy problem \eqref{DE}. 
Then the following space-time integrability (called Strichartz estimates) 
\begin{equation}\label{classtr}
\|e^{-it(\mathcal{D}+\beta)}f\|_{L_t^q L_x^r}
\lesssim
\|f\|_{H^{\sigma}}
\end{equation}
holds with 
\begin{equation}\label{range}
2\leq q\leq\infty,\quad 2\leq r\leq6,\quad
\frac{2}{q}+\frac{3}{r}=\frac{3}{2},
\quad\sigma\geq \frac1q-\frac1r+\frac12. 
\end{equation}
See the Appendix in \cite{DF} for more details.

In this paper we are concerned with obtaining those estimates in Wiener amalgam spaces\footnote{The spaces were first introduced by Feichtinger \cite{F} and have already appeared as a technical tool in the study of partial differential equations (\cite{T}).
See also \cite{S}.} which, unlike the $L^p$ spaces, control the local regularity of a function and its decay at infinity separately.
This separability makes it possible to perform a finer analysis of the local and global behavior of the flow.
These aspects were originally pointed out in the several works by Cordero and Nicola 
 (\cite{CN, CN2, CN3}) in the context of the Schr\"{o}dinger flow $e^{it\Delta}$, 
 and the works were generalized in \cite{KKS} to initial data with regularity. 
 See also \cite{KKS2} for the wave flow $e^{it\sqrt{-\Delta}}$.

As such, the results known so far are for the flows that work in the non-relativistic
regime. In contrast, we aim
here to study the Dirac flow which governs the physical system in the case of relativistic fields.
Since the operator $\mathcal{D}+\beta=-i\alpha\cdot \nabla+\beta$ mixes the components of $u$ when it acts on, it is harder to analyze the Dirac flow directly.
To overcome this difficulty, we first reduce the matter for $e^{-it(\mathcal{D}+\beta)}$ to the level of $e^{{\pm} it\sqrt{1-\Delta}}$ by using some projections onto the two-dimensional eigenspaces of the operator $\mathcal{D}+\beta$ whose symbol is $\alpha \cdot \xi + \beta$.
Indeed we note that the eigenvalues of $\alpha \cdot \xi + \beta$ are $\pm\langle\xi\rangle$ since $(\alpha\cdot\xi+\beta)^2=\langle\xi\rangle^2I$ from \eqref{matid}. (See Section \ref{sec2} for details.) 
Then we consider the problem of obtaining a pointwise estimate for the integral kernel of the Fourier multiplier $\sqrt{1-\Delta}^{-\sigma} e^{{\pm} it\sqrt{1-\Delta}}$ and of carefully estimating time-decay estimates on Wiener amalgam spaces based on the pointwise estimates.

The procedure in this paper also works for the massless case $m=0$ with the wave propagator rather than the Klein-Gordon propagator. This is because the spinorial structure of Dirac propagator is not harmful any more for obtaining the decay estimates by virtue of the projections. It is worth noticing that some relevant decay estimates were already known in \cite{KKS2} for the wave propagator. This  would yield the results with some modification for the massless case. So we shall not pursue this case here.

Before stating our results, we recall the definition of Wiener amalgam spaces.
Let $\varphi \in C_0^{\infty}$ be a test function satisfying $\|\varphi\|_{L^2} =1$.
Let $1\le p,q \le \infty$.
Then the Wiener amalgam space $W(p,q)$
is defined as the space of functions $f \in L_{\textrm{loc}}^p$
equipped with the norm
\begin{equation*}
\|f\|_{W(p,q)}= \big\|\, \|f \tau_x \varphi\|_{L^p} \big\|_{L_x^q},
\end{equation*}
where $\tau_x \varphi(\cdot) = \varphi(\cdot-x)$.
Here different choices of $\varphi$ generate the same space and yield equivalent norms.
This space can be seen as a natural extension of $L^p$ space in view of $W(p,p)=L^p$,
and $W(A,B)$ for Banach spaces $A$ and $B$ is also defined in the same way.
We also list some basic properties of these spaces which will be frequently used in the sequel:
\begin{itemize}
\item
\textit{Inclusion}; if\, $p_0 \ge p_1$ and $q_0 \le q_1$,
\begin{equation} \label{inclusion}
W(p_0, q_0) \subset W(p_1, q_1).
\end{equation}
\item
\textit{Convolution}\footnote{More generally, if\, $A_0 * A_1 \subset A$ and $B_0 * B_1 \subset B$,
$W(A_0 , B_0) * W(A_1 , B_1) \subset W(A, B).$}; if $1/p + 1 = 1/p_0 + 1/p_1$ and $1/q + 1 = 1/q_0 + 1/q_1$,
\begin{equation}\label{y-ineq}
W(p_0, q_0) * W(p_1, q_1) \subset W(p, q).
\end{equation}
\item
\textit{Interpolation}\footnote{For $0<\theta<1$, $(\cdot\,,\cdot)_{[\theta]}$ denotes the complex interpolation functor and $(1/p_\theta,1/q_\theta)$ is usually given as
$1/p_\theta = \theta/p_0 + (1-\theta)/p_1$ and  $1/q_\theta = \theta/q_0 + (1-\theta)/q_1$.};
if\, $q_0 < \infty$ or $q_1 < \infty$,
\begin{equation} \label{inter}
(W(p_0, q_0), W(p_1, q_1))_{[\theta]} = W(p_\theta, q_\theta).
\end{equation}
\end{itemize}
We refer to \cite{F,F2,F3,H} for details.

Our main results are now stated as follows.

\begin{thm}\label{thm1}
Let $\sigma>1$. Assume that $2\le \widetilde q\le\infty$, $2<q<\infty$ and $6<r\le\widetilde{r}<\infty$.
Then we have
\begin{equation}\label{T}
\|e^{-it(\mathcal{D}+\beta)}f\|_{W({\widetilde q} ,q)_t W({\widetilde r} ,r)_x} \lesssim \|f\|_{H^{\sigma}}
\end{equation}
if
\begin{equation*}
\frac{1}{\widetilde q}+\frac{3}{\widetilde r}>\frac{3}{2}-\sigma  \quad \text{and} \quad \frac{1}{q}+\frac{3}{r}=\frac{1}{2}.
\end{equation*}
\end{thm}

\begin{rem}
Roughly speaking, the estimate \eqref{T} shows that the $W({\widetilde r}, r)_x$-norm of the flow has a $L_t^q$-decay at infinity.
Since the $L_x^r$ norm in the classical type \eqref{classtr}  is rougher than the $W(\widetilde r,r)_x$-norm when $r\le\widetilde r$ (see \eqref{inclusion}), 
it is worth trying to obtain \eqref{T} especially when $r < \widetilde{r}$.
It is much more delicate. 
\end{rem}

Particularly when $\widetilde{q}=q$ and $\widetilde{r}=r$, \eqref{T} becomes equivalent to \eqref{classtr} with 
\begin{equation}\label{range2}
	\quad 2<q<\infty,\quad 6<r<\infty,\quad
	\frac{1}{q}+\frac{3}{r}=\frac{1}{2}, \quad \sigma>1.
\end{equation}
Since $r>6$, this provides some new regions of $(1/q,1/r,\sigma)$ for which \eqref{classtr} holds; the region of $(1/q, 1/r, \sigma)$ 
for \eqref{range} corresponds to a plane formed by sliding the segment $CD$ in the $\sigma$-axis direction in Figure \ref{figure},
while the corresponding region for \eqref{range2} is a plane formed by sliding the segment $AB$ in the same direction. Notice that these two planes do not obviously intersect each other. 

From complex interpolation (see \eqref{inter}) between \eqref{T} and the classical estimates \eqref{classtr}, one can obtain further estimates. We omit the details.  
One can also trivially increase $q,r$ and diminish $\widetilde q, \widetilde r$ in \eqref{T} using the inclusion relation \eqref{inclusion}.
Finally we would like to refer the reader to \cite{KN,Tr} for some (different) estimates for the Dirac equation in modulation and Wiener amalgam spaces.


\begin{figure}\label{figure}
\begin{center}
\begin{tikzpicture}[scale=5,rotate around y=-45]
\coordinate (A) at (1,0,1/6); 
\coordinate (B) at (1,1/2,0);
\coordinate (C) at (0,0,1/2);
\coordinate (D) at (5/6,1/2,1/6);

\fill[gray!50] (A) -- (B) -- (1.37,1/2,0) -- (1.37,0,1/6); 
\fill[gray!20] (C) -- (D) -- (1.2,1/2,1/6) -- (1.2,0,1/2); 

\draw[->] (0,0,0) -- (1.25,0,0); 
\draw[->] (0,0,0) -- (0,.6,0); 
\draw[->] (0,0,0) -- (0,0,.6); 
\coordinate (x) at (1.25,0,0); 
\coordinate (y) at (0,.6,0); 
\coordinate (z) at (0,0,.6); 
\node [below] at (x) {$\sigma$}; 
\node [left] at (y) {$\frac1q$}; 
\node [left] at (z) {$\frac1r$}; 

\node [below] at (0,0,1/2) {$\frac{1}{2}$}; 
\node [right] at (0,1/2+0.019,0) {$\frac{1}{2}$}; 
\node [right] at (1-0.03,0.02,0) {$1$};
\node [above] at (5/6,-0.01,0) {$\frac56$}; 
\node [below] at (-0.009,0,1/6){$\frac16$}; 

\draw (0,0,1/2) -- (5/6,1/2,1/6); 
\draw (1,0,1/6) -- (1,1/2,0); 
\draw[dotted] (5/6,0,1/2) -- (5/6,0,0); 
\draw[dotted] (1,0,1/6) -- (0,0,1/6); 
\draw[dotted] (5/6,0,1/6) -- (5/6,1/2,1/6); 
\draw[dotted] (5/6,0,1/2) -- (5/6,1/2,1/2); 
\draw[dotted] (5/6,1/2,1/6) -- (5/6,1/2,1/2); 
\draw[dotted] (5/6,1/2,1/2) -- (0,1/2,1/2); 
\draw[dotted] (0,1/2,1/2) -- (0,0,1/2); 
\draw[dotted] (1,1/2,1/6) -- (0,1/2,1/6); 
\draw[dotted] (0,1/2,1/2) -- (0,1/2,1/6); 

\draw[dotted] (0,1/2,1/6) -- (0,0,1/6); 
\draw[dotted] (0,1/2,0) -- (0,1/2,1/6); 

\draw[dotted] (1,0,0) -- (1,1/2,0); 
\draw[dotted] (1,0,0) -- (1,0,1/6); 
\draw[dotted] (1,0,1/6) -- (1,1/2,1/6); 
\draw[dotted] (1,1/2,0) -- (1,1/2,1/6); 
\draw[dotted] (0,1/2,0) -- (1,1/2,0); 

\draw (C) -- (1.2,0,1/2); 
\draw (D) -- (1.2,1/2,1/6); 
\draw (A) -- (1.37,0,1/6);  
\draw (B) -- (1.37,1/2,0);

\node [above] at (0,0,1/2+0.028) {$C$};
\node [above] at (5/6+0.04,1/2-0.008,1/6-0.015) {$D$};
\node [below] at (1-0.04,0,1/6+0.008) {$A$};
\node [above] at (1+0.04,1/2-0.008,0-0.015) {$B$};

\end{tikzpicture}
\end{center}
\caption{The ranges of $(1/q,1/r,\sigma)$ for \eqref{range} and \eqref{range2}}
\end{figure}
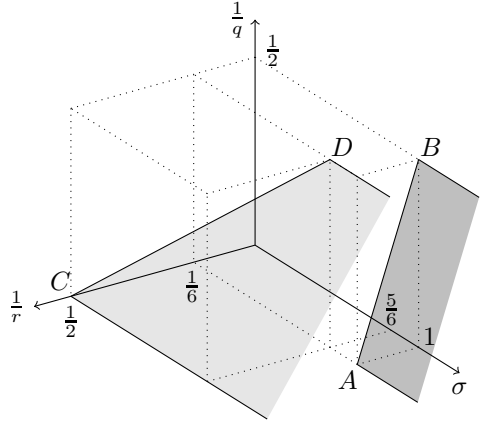


\

The rest of the paper is organized as follows.
In Section \ref{sec2} we prove Theorem \ref{thm1} using Proposition \ref{prop1} that contains the corresponding space-time estimates for the flows $e^{\pm it\sqrt{1-\Delta}}$. This proposition is proved in the same section by assuming Lemma \ref{fix} which provides time-decay estimates for the integral kernel of the Fourier multiplier 
$\sqrt{1-\Delta}^{-\sigma} e^{\pm it\sqrt{1-\Delta}}$. To obtain such estimates in Sections \ref{sec3} and \ref{sec4}, we split the kernel by dyadic (Littlewood-Paley) decomposition and use its pointwise estimates.

Throughout this paper, we use the notation $\langle \cdot \rangle = \sqrt {1+|\cdot|^2}$.
The letter $C$ stands for a positive constant which may be different
at each occurrence.
We also denote $A\lesssim B$ and $A\sim B$ to mean $A\leq CB$ and $CB \leq A \leq CB$, respectively, with unspecified constants $C>0$.

\section{Reformulation of the Dirac flow} \label{sec2}
In this section we prove Theorem \ref{thm1}. 
Let us first define the projection operators $\Pi_{\pm}(D)$ by 
\begin{equation*}
\Pi_{\pm}(D)=\frac{1}{2}\Big(I\pm\frac{1}{\langle D \rangle}(-i\alpha\cdot\nabla+\beta)\Big), \quad D=-i\nabla,
\end{equation*}
which was introduced in \cite{BH}.
By \eqref{matid} we then see that
$$-i\alpha\cdot \nabla+\beta= \langle D \rangle (\Pi_+(D)-\Pi_-(D))$$
and
\begin{equation}\label{orth}
\Pi_{\pm}(D)\Pi_{\pm}(D)=\Pi_{\pm}(D), \quad \Pi_{\pm}(D)\Pi_{\mp}(D)=0.
\end{equation}
Denoting $u_{\pm}=\Pi_{\pm}(D)u$ and 
applying $\Pi_{\pm}(D)$ to \eqref{DE}, we therefore obtain the following system of equations
\begin{equation*}
\begin{cases}
-i\partial_t u_{\pm} \pm \langle D \rangle u_{\pm}=0,\\
u_{\pm}(x,0)=\Pi_{\pm}(D)f= f_{\pm}
\end{cases}
\end{equation*}
whose solutions are 
\begin{equation*}
u_{\pm}=e^{\mp it \langle D\rangle}f_{\pm}=\frac{1}{(2\pi)^3}\int_{\mathbb{R}^{3}}e^{i(x\cdot\xi\mp t\langle\xi\rangle)}\widehat{f_{\pm}}(\xi)d\xi.
\end{equation*}
By this reformulation process, the desired estimate \eqref{T} follows immediately from the following proposition using $\|f\|_{H^{\sigma}} \sim \| f_{+}\|_{H^{\sigma}} + \|f_{-}\|_{H^{\sigma}}$ by \eqref{orth} and $u=u_++u_-$.

\begin{prop}\label{prop1}
Let $\sigma>1$. Assume that $2\le \widetilde q \le \infty$, $2<q<\infty$ and $6<r\le\widetilde{r}<\infty$.
Then we have
\begin{equation*}
\|e^{\mp it\langle D \rangle} f\|_{W({\widetilde q} ,q)_t W({\widetilde r} ,r)_x} \lesssim \|f\|_{H^{\sigma}}
\end{equation*}
if
\begin{equation}\label{c1}
\frac{1}{\widetilde q}+\frac{3}{\widetilde r}>\frac{3}{2}-\sigma \quad \text{and} \quad \frac{1}{q}+\frac{3}{r}=\frac{1}{2}.
\end{equation}
\end{prop}

In the rest of this section, we prove Proposition \ref{prop1}.
By $t\rightarrow -t$ we may consider  $e^{it\langle D\rangle}$ only.
We can apply the standard $TT^*$ argument because of the H\"older's type inequality
\begin{equation*}
|\langle F,G \rangle_{L^2_{x,t}}|
\leq \|F\|_{W(\widetilde q, q)_t W(\widetilde r, r)_x} \|G\|_{W({\widetilde q}' , q')_t W({\widetilde r}', r')_x}
\end{equation*}
which can be proved directly from the definition of Wiener amalgam spaces.
Hence it suffices to show that
\begin{equation}\label{TT*}
\bigg\| \int_\mathbb{R} \langle D \rangle^{-2\sigma} e^{i(t-s)\langle D \rangle}  F(\cdot,s) ds \bigg\|_{W(\widetilde q , q)_t W(\widetilde r , r)_x}
\lesssim \|F\|_{W(\widetilde q' , q')_t W(\widetilde r' , r')_x}.
\end{equation}
To show this, we first write the integral kernel of the Fourier multiplier $\langle D\rangle^{-\gamma} e^{it\langle D\rangle}$ as
\begin{equation}\label{ker}
K_{\gamma}(x,t):=\frac{1}{(2\pi)^3}\int_{\mathbb{R}^3} e^{i(x\cdot\xi+t\langle\xi\rangle)}\langle\xi\rangle^{-\gamma}d\xi.
\end{equation}
Then \eqref{TT*} is rephrased as follows:
\begin{equation}\label{conv}
\bigg\| \int_{\mathbb{R}} (K_{2\sigma}(\cdot,t-s) * F(\cdot,s))(x) ds \bigg\|_{W(\widetilde q ,q)_t W(\widetilde r,r)_x} \lesssim \|F\|_{W(\widetilde q' ,q')_t W(\widetilde r',r')_x}.
\end{equation}
By Minkowski's inequality and the convolution relation \eqref{y-ineq}
it follows now that
\begin{align}\label{bfHLS}
\nonumber \bigg\| \int_{\mathbb{R}} (K_{2\sigma} (\cdot,t-s)* &F(\cdot,s))(x) ds \bigg\|_{W(\widetilde q ,q)_t W(\widetilde r ,r)_x} \\
\nonumber
&\le \bigg\| \int_{\mathbb{R}} \|K_{2\sigma}(\cdot,t-s) * F(\cdot,s)\|_{W(\widetilde r ,r)_x} ds \bigg\|_{W(\widetilde q,q)_t} \\
&\le \bigg\| \int_{\mathbb{R}} \|K_{2\sigma}(\cdot,t-s)\|_{W(\frac{\widetilde r}{2} ,\frac{r}{2})_x} \|F(\cdot,s)\|_{W(\widetilde r',r')_x} ds\bigg\|_{W(\widetilde q ,q)_t}.
\end{align}
Recall the Hardy-Littlewood-Sobolev fractional integration theorem (see e.g. \cite{St}, p. 119) in dimension $1$:
\begin{equation}\label{hls}
L^{\frac{1}{\alpha},\infty}\ast L^p\hookrightarrow L^q
\end{equation}
for $0<\alpha < 1$ and $1\le p<q<\infty$ with $\frac{1}{q} +1 = \frac{1}{p} + \alpha$.
Using \eqref{hls} with $\alpha = \frac{2}{q}$ and $p=q'$
and the usual Young's inequality, the convolution relation again gives
\begin{equation*}
W(L^{\frac{\widetilde q}2}, L^{\frac q2, \infty})_t * W(\widetilde q', q')_t \subset W(\widetilde q, q )_t
\end{equation*}
for $2\leq\widetilde{q}\leq\infty$ and $2<q<\infty$.
Hence we get
\begin{align}\label{bfHLS2}
\nonumber\bigg\| \int_{\mathbb{R}} \|K_{2\sigma}(\cdot,t-s)\|_{W(\frac{\widetilde r}{2} ,\frac{r}{2})_x} &\|F(\cdot,s)\|_{W(\widetilde r' ,r')_x} ds\bigg\|_{W(\widetilde q ,q)_t}\\
\lesssim &
\| K_{2\sigma} \|_{{W(L^{\frac{\widetilde q}{2}} ,  L^{\frac q2,\infty})_t}{W(\frac{\widetilde r}{2} ,\frac{r}{2})_x}}
\|F\|_{W(\widetilde q' ,q')_t W(\widetilde r',r')_x}.
\end{align}
Combining \eqref{bfHLS} and \eqref{bfHLS2}, we finally obtain the desired estimate \eqref{conv} if
\begin{equation}\label{t}
\| K_{2\sigma} \|_{{W(L^{\frac{\widetilde q}{2}}, L^{\frac q2,\infty})_t}{W(\frac{\widetilde r}{2},\frac{r}{2})_x}} < \infty
\end{equation}
for $(\widetilde q , \widetilde r)$ and $(q,r)$ satisfying the same conditions as in Proposition \ref{prop1}.

To show \eqref{t}, we use the following time-decay estimates for the integral kernel \eqref{ker}
which will be obtained in later sections.
\begin{lem}\label{fix}
Let $\gamma>2$. Assume that $6<r\le\widetilde{r}<\infty$
and $\widetilde r\neq6/(3-\gamma)$.
Then we have
\begin{equation}\label{tdecay}
\|K_{\gamma}(\cdot,t)\|_{W(\frac{\widetilde r}{2},\frac{r}{2})_x}
\lesssim
\begin{cases}
|t|^{-1+6/r}
\quad \text{if} \quad |t|\ge1,\\
\max\{1,|t|^{-3+\gamma+6/\widetilde{r}}\}
\quad \text{if} \quad |t|\le1.
\end{cases}
\end{equation}
\end{lem}
We shall only explain the use of the case where  $\gamma \in (2,3)$ and $$\max\{1,|t|^{-3+\gamma+6/\widetilde{r}}\}=|t|^{-3+\gamma+6/\widetilde{r}}$$ 
when $|t|\le 1$, since the other cases can be used in the same way. 
We set $h(t)= \|K_{2\sigma}(\cdot,t)\|_{W(\frac{\widetilde r}{2} ,\frac{r}{2})_x}$
and choose $\varphi(t) \in C_0^{\infty}(\mathbb{R})$ supported on $\{t\in\mathbb{R}:|t|\le1\}$.
To calculate $\| h \|_{{W(L^{\frac{\widetilde q}2}, L^{\frac{q}{2},\infty})_t}}$ using \eqref{tdecay},
we handle $\|h \tau_k \varphi \|_{L_t^{\widetilde q/2}}$ dividing cases into $|k|\le 1$, $1\le|k|\le 2$ and $|k|\ge2$.
First we consider the case $|k|\le 1$.
By using \eqref{tdecay} with $\gamma=2\sigma$ and the support condition of $\varphi$,
\begin{equation}\label{k1}
\|h \tau_k \varphi \|_{L_t^{\tilde q/2}}^{\widetilde q/2}
\lesssim \int_{0<|t|\le1} |t|^{\frac{\widetilde q}{2}(-3+2\sigma+\frac{6}{\widetilde{r}})} dt + \int_{1\le|t|\le|k|+1} |t|^{\frac{\widetilde q}{2}(-1+\frac{6}{r})} dt.
\end{equation}
Since $\frac{\widetilde q}{2}(-3+2\sigma+\frac{6}{\widetilde{r}})+1>0$ by the first condition in \eqref{c1}, the first integral in \eqref{k1} is trivially finite. The second integral is bounded as follows:
\begin{align}\label{a}
\int_{1\le|t|\le|k|+1} |t|^{\frac{\widetilde q}{2}(-1+\frac{6}{r})} dt
&\lesssim\frac{(|k|+1)^{\frac{\widetilde q}{2}(-1+\frac{6}{r})+1} - 1 } {\frac{\widetilde q}{2}(-1+\frac{6}{r})+1}\nonumber\\
&\lesssim |k|.
\end{align}
Indeed, since $\frac{\widetilde q}{2}(-1+\frac{6}{r})<0$ from the condition $r>6$, the second inequality in \eqref{a} follows easily from the mean value theorem. Hence we get 
\begin{equation*}
\|h \tau_k \varphi \|_{L_t^{\widetilde q/2}}^{\widetilde q/2} \lesssim 1
\end{equation*}
when $|k|\le 1$.
The other cases $1\le|k|\le 2$ and $|k|\ge2$ are handled in the same way:
\begin{align*}
\|h \tau_k \varphi \|_{L_t^{\tilde q/2}}^{\widetilde q/2}
&\lesssim \int_{|k|-1\le|t|\le1} |t|^{\frac{\widetilde q}{2}(-3+2\sigma+\frac{6}{\widetilde{r}})} dt + \int_{1\le|t|\le|k|+1} |t|^{\frac{\widetilde q}{2}(-1+\frac{6}{r})} dt\nonumber\\
&\lesssim 1+|k|\nonumber\\
&\lesssim 1
\end{align*}
when $1\le|k|\le2$, and when $|k|\ge2$,
\begin{align*}
\|h \tau_k \varphi \|_{L_t^{\tilde q/2}}^{\widetilde q/2}
&\lesssim
\int_{|k|-1\le|t|\le|k|+1} |t|^{\frac{\widetilde q}{2}(-1+\frac{6}{r})}dt\nonumber\\
&\lesssim \frac{(|k|+1)^{\frac{\widetilde q}{2}(-1+\frac{6}{r})+1} - (|k|-1)^{\frac{\widetilde q}{2}(-1+\frac{6}{r})+1} } {\frac{\widetilde q}{2}(-1+\frac{6}{r})+1}\nonumber\\
&\lesssim (|k|-1)^{\frac{\widetilde q}{2}(-1+\frac{6}{r})}.
\end{align*}
Consequently, we get
\begin{equation} \label{local_t}
\|h\tau_k\varphi\|_{L_t^{\widetilde q/2}} \lesssim
\begin{cases}
1 \quad\textit{if}\quad |k| \leq 2,\\
(|k|-1)^{-(1-\frac{6}{r})} \quad\textit{if}\quad |k| \geq 2.
\end{cases}
\end{equation}
By \eqref{local_t}, $\|h\tau_k\varphi\|_{L_t^{\widetilde q/2}}$ belongs to $L^{\frac{q}{2}, \infty}_k$ since we are assuming the second condition in \eqref{c1} which is equivalent to $\frac{2}{q}=1-\frac{6}{r}$. This finally implies
\begin{equation*}
\| h \|_{{W(L^{\frac{\widetilde q}{2}}, L^{\frac{q}{2},\infty})_t}}<\infty
\end{equation*}
for $(\widetilde q , \widetilde r)$ and $(q,r)$ satisfying the same conditions as in Proposition \ref{prop1}
except for the case $\widetilde r=\frac{6}{3-\gamma}$.
But this case can be shown by just interpolating between the cases $\widetilde{r}=\frac{6}{3-\gamma}-\varepsilon$ and $\widetilde{r}=\frac{6}{3-\gamma}+\varepsilon$, 
with a sufficiently small $\varepsilon>0$.

\section{Pointwise estimates} \label{sec3}
In this section we obtain the following pointwise estimates for the integral kernel \eqref{ker}, which will be used for the proof of the time-decay estimates (Lemma \ref{fix}) in the next section:

\begin{lem}\label{kerlem}
Let $\gamma>2$. Then
\begin{equation}\label{xt}
|K_\gamma(x,t)|\lesssim
\begin{cases}
\max\big\{1,|(x,t)|^{\gamma-3}\big\} \quad \text{if} \quad |(x,t)|\lesssim 1,\\
|(x,t)|^{-1} \quad \text{if} \quad |(x,t)|\gtrsim 1.
\end{cases}
\end{equation}
\end{lem}

\begin{proof}
Let $\rho\in C_0^{\infty}(-2,2)$ be a compactly supported smooth function such that $\rho(s)=1$ for $|s|\le1$. For $k\in\mathbb{Z}$ we set
\begin{equation*}
\chi_k(\xi) := \rho(2^{-(k+1)}|\xi|)-\rho(2^{-(k-2)}|\xi|)
\end{equation*}
whose support is $\{\xi\in\mathbb{R}^{n} : 2^{k-2}< |\xi| \le 2^{k+2}\}$,
and consider the localized kernel $K_\gamma^{k}$ to frequencies of size $2^k$ as
\begin{equation*}
K_\gamma^{k}(x,t) = \int_{\mathbb{R}^{3}} e^{ix\cdot\xi} e^{it\langle \xi \rangle}\langle \xi \rangle^{-\gamma} \chi_k(\xi) d\xi
\end{equation*}
where $\gamma\geq0$.
Then we have
\begin{equation}\label{loc1}
|K_\gamma^k(x,t)| \lesssim 2^{3k}(1+2^{2k}|(x,t)|)^{-\frac{3}{2}}
\end{equation}
for $k\lesssim 1$, and for $k\gtrsim 1$
\begin{equation}\label{high1}
|K_\gamma^k(x,t)| \lesssim 2^{k(3-\gamma)}(1+2^k|(x,t)|)^{-1} \min\{1, 2^{k}(1+2^{k}|(x,t)|)^{-\frac{1}{2}}\}.
\end{equation}
This can be immediately derived from Lemma 2.2 in \cite{BH}
which is the case $\gamma=0$.

Summing these localized estimates over $k$ yields now \eqref{kerlem}. 
Indeed, from \eqref{loc1}, one can see that
\begin{equation*}
\sum_{k<1} |K_\gamma^k(x,t)| 
\lesssim
\sum_{k<1} 2^{3k} 
\lesssim1
\end{equation*}
when $|(x,t)|\lesssim 1$, 
and when $|(x,t)|\gtrsim 1$
\begin{align*}
\sum_{k<1} |K_\gamma^k(x,t)|
&\lesssim
\sum_{\substack{k<1\\2^{2k}|(x,t)|\lesssim 1}} 2^{3k} 
+ \sum_{\substack{k<1\\2^{2k}|(x,t)| \gtrsim 1}} |(x,t)|^{-\frac{3}{2}}\\
&\lesssim 
|(x,t)|^{-\frac{3}{2}}+|(x,t)|^{-\frac{3}{2}}\log_{2}{|(x,t)|^{\frac{1}{2}}}\\
&\lesssim
|(x,t)|^{-1}.
\end{align*}
In summary, 
\begin{equation}\label{less1sum}
\sum_{k<1} |K_\gamma^k(x,t)|
\lesssim
\begin{cases}
1 \quad \text{if} \quad |(x,t)|\lesssim 1,\\
|(x,t)|^{-1} \quad \text{if} \quad |(x,t)|\gtrsim 1.
\end{cases}
\end{equation}
On the other hand, from \eqref{high1},
one can see that 
\begin{align*}
\sum_{k\ge1} |K_\gamma^k(x,t)|
&\lesssim
\sum_{\substack{k\ge1 \\ 2^k|(x,t)|\lesssim 1}} 
2^{k(3-\gamma)} + \sum_{\substack{k\ge1 \\ 2^k|(x,t)|\gtrsim 1}} 
2^{k(2-\gamma)}|(x,t)|^{-1}\\
&\lesssim
|(x,t)|^{\gamma-3} \quad \text{if} \quad \gamma>2
\end{align*}
when $|(x,t)|\lesssim 1$,
and when $|(x,t)|\gtrsim 1$
\begin{align*}
\sum_{k\ge1} |K_\gamma^k(x,t)|
&\lesssim
\sum_{\substack{k\ge1 \\ |(x,t)|\gtrsim 2^k}} 
2^{k(\frac{5}{2}-\gamma)}|(x,t)|^{-\frac{3}{2}} + \sum_{\substack{k\ge1\\|(x,t)|\lesssim 2^k}} 
2^{k(2-\gamma)}|(x,t)|^{-1} \\
&\lesssim
|(x,t)|^{1-\gamma} \quad \text{if} \quad \gamma>2.
\end{align*}
In summary, for $\gamma>2$,
\begin{equation}\label{gtr1sum}
\sum_{k\ge 1} |K_\gamma^k(x,t)|
\lesssim
\begin{cases}
|(x,t)|^{\gamma-3} \quad \text{if} \quad |(x,t)|\lesssim 1,\\
|(x,t)|^{1-\gamma} \quad \text{if} \quad |(x,t)|\gtrsim 1.
\end{cases}
\end{equation}
Finally, \eqref{xt} follows from combining \eqref{less1sum} and \eqref{gtr1sum}.
\end{proof}

\section{Time-decay estimates}\label{sec4}

We now finish the proof of Lemma \ref{fix} by making use of the pointwise estimates just obtained in the previous section.

We first choose $\varphi(x) \in C_0^{\infty}(\mathbb{R}^3)$ supported on $\{x\in\mathbb{R}^3:|x|\le1\}$
to calculate
\begin{equation}\label{fixedes1}
\|K_{\gamma}(\cdot,t)\|_{W(\frac{\widetilde r}{2},\frac{r}{2})_x}^{r/2}
=\int_{\mathbb{R}^3}  \|K_{\gamma}(\cdot,t) \tau_y \varphi(\cdot)\|_{L_x^{\widetilde r/2}}^{\frac{r}{2}} dy,
\end{equation}
and set
\begin{equation*}
I_t(|y|):=\int_{|y|-1 \le |x| \le |y|+1} |K_{\gamma}(x,t)|^{{\widetilde r}/2} dx.
\end{equation*}
By the size of $\textrm{supp}\,\varphi$, $\|K_{\gamma}(\cdot,t) \tau_y  \varphi(\cdot) \|_{L_x^{\widetilde r/2}}^{\widetilde r/2}$
is then calculated as
\begin{equation*}
\int_{|x-y| \le 1} |K_{\gamma}(x,t)|^{{\widetilde r}/2} dx
\lesssim
\begin{cases}
I_t(|y|)\quad\text{if} \quad |y|\leq1,\\
|y|^{-2}I_t(|y|)\quad\text{if} \quad |y|\ge1
\end{cases}
\end{equation*}
because the region $\{x\in\mathbb{R}^3:|y|-1 \le |x| \le |y|+1\}$ when $|y| \ge 1$
contains balls with radius $1$ as many as a constant multiple of $|y|^{2}$.
Finally, we calculate \eqref{fixedes1} as
\begin{align}\label{secint1}
\nonumber\|K_{\gamma}(\cdot,t)\|_{W(\frac{\widetilde r}{2},\frac{r}{2})_x}^{r/2}
&\lesssim\int_{|y| \le 1} I_t(|y|)^{\frac{r}{\widetilde{r}}}dy
+\int_{|y|\ge1}   \big(|y|^{-2}I_t(|y|)\big)^{\frac{r}{\widetilde{r}}}dy\\
&:=A+B.
\end{align}

We only consider the case $2<\gamma<3$.
The case $\gamma\ge3$ is more simply handled in the same way as well
and so we shall omit the details.  
Comparing the size of $|t|$, $|x|$ and $1$, we begin by deducing from \eqref{xt} that
\begin{equation}\label{upbd1}
|K_{\gamma}(x,t)|\lesssim
\begin{cases}
|t|^{-1} \quad \text{if} \quad \max\{|t|,|x|,1\}=|t|,\\
\min\{|x|^{-3+\gamma},|t|^{-3+\gamma}\} \quad \text{if} \quad \max\{|t|,|x|,1\}=1,\\
|x|^{-1} \quad \text{if} \quad \max\{|t|,|x|,1\}=|x|.
\end{cases}
\end{equation}
By considering both the size of $\textrm{supp}\,\varphi$ and the different behaviors of \eqref{upbd1} near $|x|=|t|$ as well as $|x|=1$,
we now estimate $I_t(|y|)$ dividing cases into:
\begin{itemize}
\item when $|t|\le1$,
\begin{equation*}
|y|\le|t|+1, \quad |t|+1\le|y|\le2, \quad |y|\ge2,
\end{equation*}
\item when $1\le|t|\le3$,
\begin{equation*}
|y|\le|t|-1, \quad |t|-1\le|y|\le2, \quad 2\le|y|\le|t|+1, \quad |y|\ge|t|+1,
\end{equation*}
\item when $|t|\ge3$,
\begin{equation*}
|y|\le2, \quad 2\le|y|\le|t|-1, \quad |t|-1\le|y|\le|t|+1, \quad |y|\ge|t|+1.
\end{equation*}
\end{itemize}

\subsection{The case $|t|\le1$}
When $|y|\le|t|+1$ firstly, we split the integral region of $I_{t}(|y|)$ into 
\begin{equation*}
|y|-1 \le |x| \le |t|,\quad |t|\le|x|\le 1 \quad\text{and}\quad 1 \le |x| \le |y|+1,
\end{equation*}
and apply the pointwise estimates \eqref{upbd1} appropriately in each region.
Finally, we carefully calculate the resulting integrals after using the polar coordinates;
\begin{itemize}
\item when $|y| \le 1$;
\begin{align} \label{a1}
I_t(|y|)
\lesssim&\;
\int_{0}^{|t|} |t|^{\frac{\widetilde{r}}{2}(-3+\gamma)} \rho^{2} d\rho 
+\int_{|t|}^{1} \rho^{\frac{\widetilde{r}}{2}(-3+\gamma)+2} d\rho
+\int_{1}^{|y|+1}\rho^{-\frac{\widetilde{r}}{2}+2}d\rho\nonumber\\
\lesssim&\;
\begin{cases}
|t|^{\frac{\widetilde{r}}{2}(-3+\gamma)+3} \quad \text{if} \quad \widetilde{r}>\frac{6}{3-\gamma},\\
1 \quad \text{if} \quad \widetilde{r}<\frac{6}{3-\gamma},
\end{cases}
\end{align}
\item when $|y|\ge 1$;
\begin{align}\label{a11}
I_t(|y|)
&\lesssim
\int_{|y|-1}^{|t|}|t|^{\frac{\widetilde{r}}{2}(-3+\gamma)}\rho^2 d\rho
+\int_{|t|}^{1} \rho^{\frac{\widetilde{r}}{2}(-3+\gamma)+2} d\rho
+\int_{1}^{|y|+1}\rho^{-\frac{\widetilde{r}}{2}+2}d\rho\nonumber\\
&\lesssim
\begin{cases}
|t|^{\frac{\widetilde r}{2}(-3+\gamma)+3} \quad \text{if} \quad \widetilde{r}>\frac{6}{3-\gamma},\\
1 \quad \text{if} \quad \widetilde{r}<\frac{6}{3-\gamma}.
\end{cases}
\end{align}
\end{itemize}

Next, when $|t|+1\le|y|\le2$ and $|y|\ge2$, we similarly and more simply obtain 
\begin{align} \label{b1}
I_t(|y|)&\lesssim \int_{|y|-1}^{1} \rho^{\frac{\widetilde{r}}{2}(-3+\gamma)+2} d\rho+\int_{1}^{|y|+1} \rho^{-\frac{\widetilde{r}}2+2}d\rho\nonumber\\
&\lesssim
\begin{cases}
(|y|-1)^{\frac{\widetilde r}{2}(-3+\gamma)+3} \quad \text{if} \quad \widetilde{r}>\frac{6}{3-\gamma},\\
1 \quad \text{if} \quad \widetilde{r}<\frac{6}{3-\gamma},
\end{cases}
\end{align}
and
\begin{equation}\label{c01}
I_t(|y|)
\lesssim
\int_{|y|-1}^{|y|+1} \rho^{-\frac{\widetilde r}{2}+2}d\rho
\lesssim
(|y|-1)^{-\frac{\widetilde{r}}{2}+2}\quad\text{since}\quad \widetilde{r}>4,
\end{equation}
respectively.

\subsection{The case $1\le|t|\le3$}
Similarly in the same way, 
if $|y|\le|t|-1$
\begin{itemize}
\item when $|y| \le 1$;
\begin{align}\label{d1}
I_t(|y|)&\lesssim |t|^{-\frac{\widetilde r}{2}} \int_{0}^{1} \rho^{2} d\rho+|t|^{-\frac{\widetilde{r}}{2}}\int_{1}^{|y|+1} \rho^{2}d\rho\nonumber\\
&\lesssim
|t|^{-\frac{\widetilde r}{2}}(|y|+1)^{3},
\end{align}
\item when $|y| \ge 1$;
\begin{align}\label{d11}
I_t(|y|)&\lesssim |t|^{-\frac{\widetilde r}{2}} \int_{|y|-1}^{1} \rho^{2} d\rho+|t|^{-\frac{\widetilde{r}}{2}}\int_{1}^{|y|+1} \rho^{2}d\rho\nonumber\\
&\lesssim
|t|^{-\frac{\widetilde r}{2}}(|y|+1)^{3},
\end{align}
\end{itemize}
and if $|t|-1\le|y|\le2$
\begin{itemize}
\item when $|y| \le 1$;
\begin{align}\label{e1}
I_t(|y|)
&\lesssim
|t|^{-\frac{\widetilde{r}}{2}}\int_{0}^{1} \rho^{2}d\rho
+|t|^{-\frac{\widetilde{r}}{2}}\int_{1}^{|t|} \rho^{2}d\rho
+\int_{|t|}^{|y|+1} \rho^{-\frac{\widetilde r}{2}+2}d\rho\nonumber\\
&\lesssim
|t|^{-\frac{\widetilde{r}}{2}+3},
\end{align}
\item when $|y| \ge 1$;
\begin{align}\label{e11}
I_t(|y|)
&\lesssim
|t|^{-\frac{\widetilde{r}}{2}}\int_{|y|-1}^{1} \rho^{2}d\rho
+|t|^{-\frac{\widetilde{r}}{2}}\int_{1}^{|t|} \rho^{2}d\rho
+\int_{|t|}^{|y|+1} \rho^{-\frac{\widetilde r}{2}+2}d\rho\nonumber\\
&\lesssim
|t|^{-\frac{\widetilde{r}}{2}+3},
\end{align}
\end{itemize}
and if $2\le|y|\le|t|+1$
\begin{align} \label{f1}
I_t(|y|)&\lesssim |t|^{-\frac{\widetilde{r}}{2}}\int_{|y|-1}^{|t|} \rho^{2}d\rho
+\int_{|t|}^{|y|+1} \rho^{-\frac{\widetilde{r}}{2}+2}d\rho\nonumber\\
&\lesssim
|t|^{-\frac{\widetilde{r}}{2}+3},
\end{align}
and finally if $|y|\ge|t|+1$
\begin{equation}\label{g1}
I_t(|y|)
\lesssim
 \int_{|y|-1}^{|y|+1} \rho^{-\frac{\widetilde r}{2}+2}d\rho
\lesssim
(|y|-1)^{-\frac{\widetilde{r}}{2}+2}.
\end{equation}

\subsection{The case $|t|\ge3$} 
Similarly, if $|y|\le2$
\begin{itemize}
\item when $|y| \le 1$;
\begin{align}\label{h1}
I_t(|y|)&\lesssim |t|^{-\frac{\widetilde r}{2}} \int_{0}^{1} \rho^{2} d\rho+|t|^{-\frac{\widetilde{r}}{2}}\int_{1}^{|y|+1} \rho^{2}d\rho\nonumber\\
&\lesssim
|t|^{-\frac{\widetilde r}{2}}(|y|+1)^{3},
\end{align}
\item when $|y| \ge 1$;
\begin{align}\label{h11}
I_t(|y|)&\lesssim |t|^{-\frac{\widetilde r}{2}} \int_{|y|-1}^{1} \rho^{2} d\rho+|t|^{-\frac{\widetilde{r}}{2}}\int_{1}^{|y|+1} \rho^{2}d\rho\nonumber\\
&\lesssim
|t|^{-\frac{\widetilde r}{2}}(|y|+1)^{3},
\end{align}
\end{itemize}
and if $2\le|y|\le|t|-1$
\begin{equation}\label{i1}
I_t(|y|)\lesssim |t|^{-\frac{\widetilde{r}}{2}}\int_{|y|-1}^{|y|+1} \rho^{2}d\rho
\lesssim |t|^{-\frac{\widetilde{r}}{2}}(|y|+1)^{3},
\end{equation}
and if $|t|-1\le|y|\le|t|+1$
\begin{align} \label{j1}
I_t(|y|)&\lesssim
|t|^{-\frac{\widetilde{r}}{2}}\int_{|y|-1}^{|t|} \rho^{2}d\rho
+\int_{|t|}^{|y|+1} \rho^{-\frac{\widetilde{r}}{2}+2}d\rho\nonumber\\
&\lesssim
|t|^{-\frac{\widetilde{r}}{2}+3},
\end{align}
and finally if $|y|\ge|t|+1$
\begin{equation}\label{k01}
I_t(|y|)
\lesssim
 \int_{|y|-1}^{|y|+1} \rho^{-\frac{\widetilde r}{2}+2}d\rho
\lesssim
(|y|-1)^{-\frac{\widetilde{r}}{2}+2}.
\end{equation}

\subsection{Fitting it all together}
We now estimate the first part $A$ in \eqref{secint1} using some of bounds obtained above, \eqref{a1}, \eqref{d1}, \eqref{e1} and \eqref{h1}, as follows:

\begin{itemize}
\item when $|t|\le1$;
\begin{equation*}
A
\lesssim
\int_{|y| \le 1} \eqref{a1}^{\frac{r}{\widetilde{r}}}dy
\lesssim
\begin{cases}
|t|^{\frac{r}{2}(-3+\gamma)+\frac{3r}{\widetilde{r}}} \quad \text{if} \quad \widetilde{r}>\frac{6}{3-\gamma},\\
1 \quad \text{if} \quad \widetilde{r}<\frac{6}{3-\gamma},
\end{cases}
\end{equation*}

\item when $1\le|t|\le2$;
\begin{equation*}
A
\lesssim
\int_{|y|\le|t|-1}  \eqref{d1}^{\frac{r}{\widetilde{r}}}dy
+\int_{|t|-1\le |y| \le 1} \eqref{e1}^{\frac{r}{\widetilde{r}}}dy
\lesssim
|t|^{-\frac{r}{2}+\frac{3r}{\widetilde r}}
\sim
|t|^{-\frac{r}{2}},
\end{equation*}

\item when $2\le|t|\le3$;
\begin{equation*}
A
\lesssim
\int_{|y|\le1} \eqref{d1}^{\frac{r}{\widetilde{r}}}dy
\lesssim
|t|^{-\frac{r}{2}},
\end{equation*}

\item when $|t|\ge3$;
\begin{equation*}
A
\lesssim
\int_{|y| \le 1} \eqref{h1}^{\frac{r}{\widetilde{r}}}dy
\lesssim
|t|^{-\frac{r}{2}}.
\end{equation*}

\end{itemize} 

Next we estimate the second part $B$ in \eqref{secint1}.
With the notation $d\mu=|y|^{-\frac{2r}{\widetilde{r}}}dy$, one can see that
\begin{itemize}
\item when $|t|\le1$;
\begin{align*}
B\lesssim&\,
\int_{1\le|y|\le|t|+1} \eqref{a11}^{\frac{r}{\widetilde r}}d\mu
+\int_{|t|+1\le|y|\le2}\eqref{b1}^{\frac{r}{\widetilde r}}d\mu
+\int_{|y|\ge2}\eqref{c01}^{\frac{r}{\widetilde r}}d\mu\\
\lesssim&\,
\begin{cases}
|t|^{\frac{r}{2}(-3+\gamma)+\frac{3r}{\widetilde{r}}} \quad \text{if} \quad \widetilde{r}>\frac{6}{3-\gamma},\\
1 \quad \text{if} \quad \widetilde{r}<\frac{6}{3-\gamma},
\end{cases}
\end{align*}

\item when $1\le|t|\le2$;
\begin{align*}
B\lesssim&
\,\int_{1\le|y| \le 2} \eqref{e11}^{\frac{r}{\widetilde{r}}}d\mu
+\int_{2\le|y| \le |t|+1} \eqref{f1}^{\frac{r}{\widetilde{r}}}d\mu
+\int_{|y| \ge |t|+1} \eqref{g1}^{\frac{r}{\widetilde{r}}}d\mu\\
\lesssim&\, 
|t|^{-\frac{r}{2}+3},
\end{align*}

\item when $2\le|t|\le3$;
\begin{align*}
B\lesssim&
\,\int_{1\le|y| \le |t|-1} \eqref{d11}^{\frac{r}{\widetilde{r}}}d\mu
+\int_{|t|-1\le|y| \le 2} \eqref{e11}^{\frac{r}{\widetilde{r}}}d\mu\\
&+\int_{2 \le |y| \le |t|+1} \eqref{f1}^{\frac{r}{\widetilde{r}}}d\mu
+\int_{|y| \ge |t|+1} \eqref{g1}^{\frac{r}{\widetilde{r}}}d\mu\\
\lesssim&\, 
|t|^{-\frac{r}{2}+3},
\end{align*}

\item when $|t|\ge3$;
\begin{align*}
B\lesssim&
\,\int_{1\le|y| \le 2} \eqref{h11}^{\frac{r}{\widetilde{r}}}d\mu
+\int_{2\le|y| \le |t|-1} \eqref{i1}^{\frac{r}{\widetilde{r}}}d\mu\\
&+\int_{|t|-1 \le |y| \le |t|+1} \eqref{j1}^{\frac{r}{\widetilde{r}}}d\mu
+\int_{|y| \ge |t|+1} \eqref{k01}^{\frac{r}{\widetilde{r}}}d\mu\\
\lesssim&\, 
|t|^{-\frac{r}{2}+3}.
\end{align*}
\end{itemize}
Each last integral when estimating $B$ above boils down to 
\begin{align*}
\int_{|y| \ge a} (|y|-1)^{(-\frac{\widetilde{r}}{2}+2){\frac{r}{\widetilde{r}}}}d\mu
&=
\int_{a-1}^{\infty}\rho^{-\frac{r}{2}+\frac{2r}{\widetilde{r}}}(\rho+1)^{-\frac{2r}{\widetilde{r}}+2}d\rho\\
&\lesssim
\int_{a-1}^{\infty}\rho^{-\frac{r}{2}+2}d\rho\\
&\lesssim
(a-1)^{-\frac{r}{2}+3},
\end{align*}
with $a=2$ if $|t|\leq1$ and $a=|t|+1$ otherwise,  
provided $\widetilde{r} \ge r > 6$.

Combining \eqref{secint1} with these bounds for $A$ and $B$, we have
\begin{equation*}
\|K_{\gamma}(\cdot,t)\|_{W(\frac{\widetilde r}{2},\frac{r}{2})_x}^{r/2}
\lesssim
\begin{cases}
|t|^{-\frac{r}{2}+3} \quad \text{if} \quad |t|\ge1, \\
\max\{1,|t|^{\frac{r}{2}(-3+\gamma)+\frac{3r}{\widetilde{r}}}\} \quad \text{if} \quad |t|\le1
\end{cases}
\end{equation*}
under the assumptions $6<r\le\widetilde r<\infty$ and $\widetilde r\neq\frac{6}{3-\gamma}$.

\subsubsection*{Acknowledgement}
The authors would like to thank the anonymous referees for their
valuable comments which helped to improve the manuscript.

\end{document}